\documentclass[11pt]{amsart}
\usepackage{amsfonts}
\usepackage{amsmath,amssymb}
\usepackage[foot]{amsaddr}
\usepackage{amsthm}
\usepackage{graphicx}
\usepackage{mathtools}
\usepackage{mathrsfs}
\usepackage{xcolor}
\usepackage{textcomp}
\usepackage{booktabs}
\usepackage{tikz}
\usepackage{enumitem}
\usepackage[colorlinks=true, linkcolor=blue, citecolor=blue, urlcolor=blue, psdextra]{hyperref}

\usetikzlibrary{arrows.meta,calc,decorations.pathreplacing}

\numberwithin{equation}{section}
\allowdisplaybreaks

\theoremstyle{plain}
\newtheorem{theorem}{Theorem}[section]
\newtheorem{proposition}[theorem]{Proposition}
\newtheorem{lemma}[theorem]{Lemma}

\theoremstyle{definition}

\newtheorem{example}{Example}[section]

\theoremstyle{remark}
\newtheorem{remark}[theorem]{Remark}

\begin{document}

\bibliographystyle{abbrv}

\renewcommand{\baselinestretch}{1.05}

\title[Large deviations for linear regions]
{Local large deviations for linear-region growth in random piecewise-linear networks}

\author{Recep {\"O}zkan}
\address{Middle East Technical University, Ankara, Turkey}
\email{rozkan@metu.edu.tr}

\author{Christian Hirsch}
\address{Department of Mathematics, Aarhus University, Aarhus, Denmark}
\email{hirsch@math.au.dk}



\begin{abstract}
We study a random compositional model for the growth of affine regions in deep
piecewise-linear networks. The model is generated by i.i.d.\ perturbations of the
symmetric height-one tent map, and the main observable is the number \(N_n\) of
affine pieces after \(n\) layers. We prove the existence of a submultiplicative
pressure for \(N_n\), yielding exponential upper bounds for both tails of
\(n^{-1}\log N_n\). The same argument applies to abstract submultiplicative
complexity observables and gives higher-dimensional extensions for
convex-polytopal affine-cover counts and worst-line affine-piece counts. Since
the true branch count has no matching supermultiplicative inequality, lower
bounds require a separate certified construction. We introduce a finite-state
defect process that records branches whose future splitting can be guaranteed,
and use bridge words to obtain constructive upper-tail lower bounds. In a
uniformly favorable small-noise regime, this process is governed by a companion
matrix whose Perron root tends to \(2\), implying eventual exclusion of lower
tails below \(\log 2-\xi\).
\end{abstract}

\maketitle

\section{Introduction}\label{sec:intro}

The geometry of linear regions is a standard proxy for complexity in deep piecewise-linear networks. For ReLU and related activations, the input space is partitioned into polyhedral cells on each of which the network computes an affine map, and a substantial literature studies how the number and arrangement of these cells depend on architecture and depth \cite{BianchiniScarselli2014,PascanuMontufarBengio2014,MontufarPascanuChoBengio2014,RaghuPooleKleinbergGanguliSohlDickstein2017,SerraTjandraatmadjaRamalingam2018}. At one end of this literature lie explicit depth-separation constructions, most notably Telgarsky's sawtooth construction, where repeated composition of a one-dimensional folding map creates exponentially many oscillations or regions \cite{Telgarsky2015,Telgarsky2016}. At the other end lie results showing that, for standard random initialization, this worst-case mechanism is highly nongeneric \cite{NovakBahriAbolafiaPenningtonSohlDickstein2018,HaninRolnick2019}.
Figure~\ref{fig:relu-regions-2d} illustrates the type of geometric complexity
measured by linear-region counts in a simple two-dimensional ReLU network. The
present paper studies a one-dimensional random compositional model in which the
same region-counting problem can be analyzed through submultiplicative pressure
and certified lower processes.

\begin{figure}[!htpb]
\centering
\includegraphics[width=.85\textwidth]{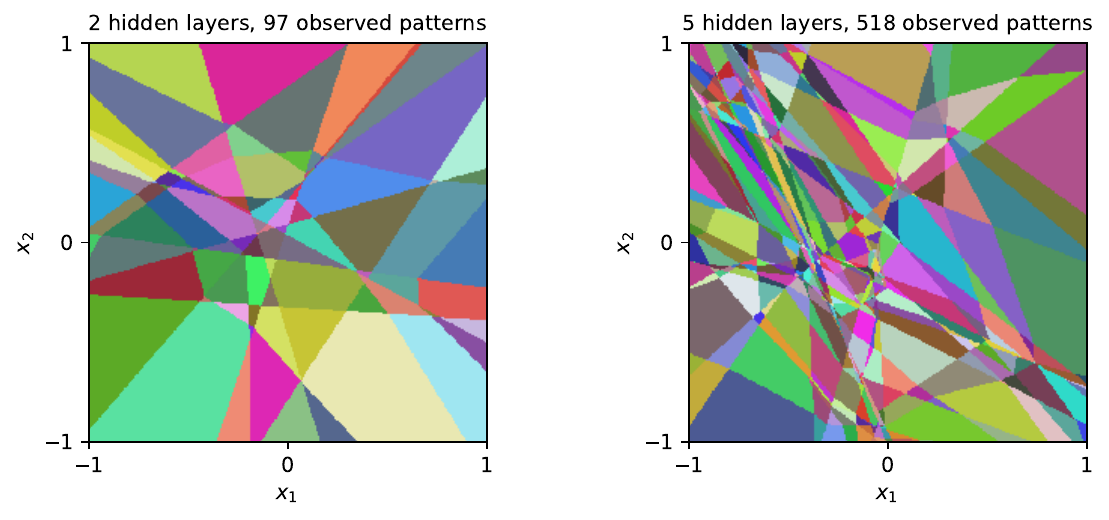}
\caption{Linear regions of randomly initialized ReLU networks on the input
square \([-1,1]^2\). Colors indicate different activation patterns, and hence
different affine regions of the realized piecewise-affine map. The left panel
uses two hidden layers, while the right panel uses five hidden layers. The
figure is illustrative only; the results of this paper concern the exact
one-dimensional branch-count process generated by random tent-map
compositions.}
\label{fig:relu-regions-2d}
\end{figure}

The purpose of the present paper is to study a model that sits between these two viewpoints. We keep the compositional folding mechanism of a one-dimensional depth-separation construction, but we randomize the layer parameters. The resulting process is no longer deterministic, and the natural observable becomes the exponential growth rate of the number of affine pieces. Our goal is to understand this growth from a large-deviation point of view.

Large-deviation principles (LDPs) for random growth observables often arise from subadditive, superadditive, or approximately additive structures \cite{SeppalainenYukich2001}. The present model exhibits a related but asymmetric structure. The branch-count admits a submultiplicative inequality, which yields pressure bounds and exponential large-deviation upper estimates. However, no corresponding supermultiplicative inequality is available for the true branch count. This asymmetry is the reason that the lower-bound theory below requires additional geometric structure.
There are substantial LDPs for Galton--Watson processes,
branching processes in random environment, and branching random walks; see, for
example, \cite{BigginsBingham1993,NeyVidyashankar2004,
FleischmannWachtel2007,BansayeBerestycki2009}. This theory, however, does not
apply directly to the affine-region count studied here. The affine pieces are
not particles with conditionally independent reproduction. They are cells of a
random piecewise-affine partition, and their future splitting depends on their
geometric position and on the common future layers shared by all cells. Thus,
the term ``branch'' is used below in a geometric sense, not in the probabilistic
sense of a Galton--Watson or branching-random-walk genealogy.

Recent work has developed LDPs for a broad range of neural-network observables and asymptotic regimes, including outputs and covariance processes of Gaussian deep networks, deep networks with ReLU or other linearly growing activations, functional output processes in wide or deep limits, and empirical weight dynamics under stochastic-gradient training of shallow networks; see, for instance, \cite{MacciPacchiarottiTorrisi2026,AndreisBassettiHirsch2026,Vogel2026,DiLilloMacciPacchiarotti2026,MacciPacchiarottiPapagiannouliTorrisiTrevisan2026,HirschWillhalm2025}. The observable studied here is of a different nature: we focus on the geometric growth of affine pieces under random composition.

The model is simple. Each layer is a random tent map \(T_{\omega_n}\), and the depth-\(n\) composition is \(F_n=T_{\omega_n}\circ\cdots\circ T_{\omega_1}\). The quantity of interest is \(N_n\), the number of affine pieces of \(F_n\) on \([0,1]\). Thus, \(N_n\) is a one-dimensional linear-region count for a random deep piecewise-linear network. Since each tent map has an explicit width-three ReLU representation, this model is directly embedded in the standard ReLU formalism. The deterministic choice \(\omega_n\equiv (2,2)\) recovers the symmetric height-one tent map underlying Telgarsky's construction and hence the classical exponential doubling mechanism.

The map \(T_{2,2}\) is the deterministic sawtooth block underlying the
1D depth-separation constructions of Telgarsky
\cite{Telgarsky2015,Telgarsky2016}; see also
\cite[Sec.~6]{PetersenZech2024} for a textbook treatment of the same doubling
mechanism in terms of piecewise-linear regions. In this deterministic case,
each composition doubles the affine-piece count, so \(N_n=2^n\). The
contribution of the present paper is not a new deterministic doubling
construction, but the large-deviation analysis of \(N_n\) when the layers are
random perturbations of this block.

The random perturbation of this construction is natural both theoretically and practically. In any realistic network, layer parameters are not tuned to the exact Telgarsky configuration \(\omega_n\equiv (2,2)\); rather, they deviate due to initialization, finite-precision arithmetic, or training dynamics. The probabilistic question therefore arises as: how stable is exponential linear-region growth under such deviations? Hanin and Rolnick \cite{HaninRolnick2019} showed that, along any fixed one-dimensional input trajectory, the average number of linear regions grows only linearly with the total number of neurons, far below the exponential worst-case bound. The present model occupies the intermediate regime: the layer parameters are random, but concentrated near \(T_{2,2}\) rather than drawn from a standard initialization. This is precisely the regime where large-deviation questions become meaningful. One can rigorously quantify how likely it is that near-maximal expressivity is achieved or lost under small perturbations of the Telgarsky configuration.

Figure~\ref{fig:perturbed-telgarsky} illustrates the corresponding
one-dimensional mechanism. Small perturbations of the Telgarsky block may
preserve the full affine-piece count at moderate depth, whereas more lossy
perturbations can cause some branches to miss later peaks and eliminate
potential teeth.

\begin{figure}[!htbp]
\centering
\includegraphics[width=.92\textwidth]{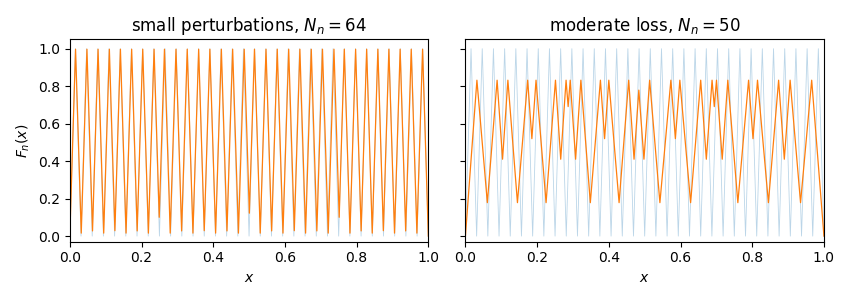}
\caption{Perturbed Telgarsky compositions at moderate depth. In each panel, the
pale curve shows the deterministic reference $n$-fold composition \(T_{2,2} \circ \cdots \circ T_{2,2}\), while the
darker curve shows one realization of the random composition \(F_n\). Small
perturbations may preserve the full affine-piece count, whereas moderate lossy
perturbations can cause some branches to miss later peaks and eliminate
potential teeth. The figure is illustrative only; the results below concern the
exact branch-count process \(N_n\), not numerical sampling of the graph.}
\label{fig:perturbed-telgarsky}
\end{figure}

The main results are stated in Section~\ref{sec:model-results}. For the upper bound, a pathwise submultiplicative inequality for \(N_n\) yields a bilateral pressure \(\Lambda(q)=\lim_{n\to\infty} n^{-1}\log \mathbb E[N_n^q]\), \(q\in\mathbb R\), and exponential upper bounds for both tails of \(n^{-1}\log N_n\). Because this argument uses only submultiplicativity and independent environment blocks, we also record an abstract pressure principle and two higher-dimensional consequences, one for convex-polytopal affine-cover counts and one for worst-line affine-piece counts.

The lower-bound theory is local and constructive. Its purpose is to compensate for the absence of a supermultiplicative inequality for the true branch count. Near the deterministic Telgarsky regime, branches may lose future splitting potential when their image intervals approach the endpoints of \([0,1]\). We therefore introduce a coarse defect dynamics that tracks only enough endpoint information to certify future splits. This gives a favorable finite-state subsystem and a bridge word that restores quasi-multiplicativity for a certified lower count. At fixed coarse parameters \((r,m)\), the bridge construction gives upper-tail lower bounds for the true observable \(N_n\) along bridge-compatible subsequences. In a uniformly favorable small-noise regime, the lower subsystem becomes deterministic and is governed by a companion matrix \(A_m\), whose Perron root \(\rho_m\) satisfies \(\rho_m\uparrow 2\). Consequently, for every \(\xi>0\), if the support of the noise is sufficiently small, then the lower tail below \(\log 2-\xi\) is eventually impossible.

The bridge construction is intentionally one-sided. It gives a sufficient mechanism for high upper-tail growth, not a sharp characterization of all high-growth trajectories. The certified lower count may be much smaller than the true number of affine pieces. The remaining obstruction to a full sharp large-deviation principle is therefore an asymptotic-exactness theorem comparing the true branch count to the certified coarse transfer process on the relevant rare-event scale.

The paper is organized as follows. Section~\ref{sec:model-results} introduces the model and states the main theorems, first the pressure upper bounds and then the certified lower-bound mechanisms.  Section~\ref{sec:pressure} proves the pressure theorem and the resulting exponential upper bounds. Section~\ref{sec:favorable} constructs the favorable subsystem, defines the coarse transfer matrices, and proves the bridge lemma. Section~\ref{sec:bridged-lower} proves the bridged lower bound for the upper tail of the original branch-count observable. Section~\ref{sec:uniform} analyzes the uniformly favorable regime, proves the Perron-root asymptotics, and proves the local theorem near \(\log 2\).

\section{Model and main results}\label{sec:model-results}

This section has three parts. Section~\ref{subsec:model} defines the random tent model and the branch-count observable. Section~\ref{subsec:ldp-upper} states the large-deviation upper bounds, which are global consequences of submultiplicativity. Section~\ref{subsec:ldp-lower} states the large-deviation lower bounds, which require a certified lower process based on coarse endpoint defects.

\subsection{Model}\label{subsec:model}

Fix \(0<\varepsilon<1\) and \(h(a,b):=ab/(a+b)\). Let \(\mu_\varepsilon\) be a probability measure supported on
\[
U_\varepsilon
:=
\bigl\{(a,b)\in(0,\infty)^2:
|a-2|\le\varepsilon,\ |b-2|\le\varepsilon,\ h(a,b)\le1
\bigr\},
\]
Let \((\omega_n)_{n\ge1}\) be i.i.d. with law \(\mu_\varepsilon\), and write \(\omega_n=(A_n,B_n)\). For \(\omega=(a,b)\in U_\varepsilon\), define
\[
	T_\omega(x):=(\min(ax,b(1-x)))_+,
\]
The map \(T_\omega\) is continuous and piecewise affine, with peak location \(c(\omega):=b/(a+b)\) and peak height \(h(\omega)=ab/(a+b)\le1\). Thus, \(T_\omega([0,1])\subseteq[0,1]\); see Figure \ref{fig:defects-splitting} for an illustration.
% Preamble:
% \usepackage{tikz}
% \usetikzlibrary{arrows.meta,calc,decorations.pathreplacing}

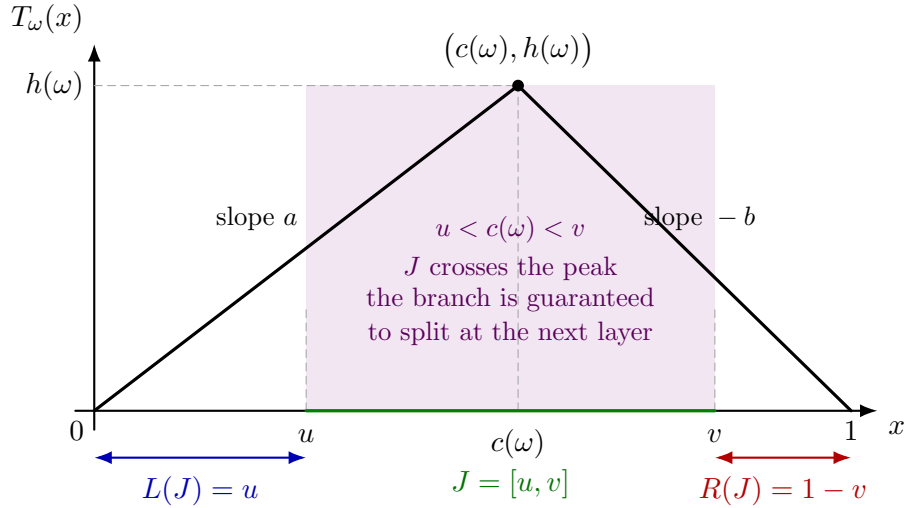
\begin{figure}[!htbp]
\centering
\begin{tikzpicture}[
    x=1cm,
    y=1cm,
    >=Latex,
    line cap=round,
    line join=round,
    axis/.style={->, thick},
    tent/.style={very thick, black},
    guide/.style={densely dashed, gray!70},
    interval/.style={very thick, green!50!black},
    leftdef/.style={thick, blue!70!black, <->},
    rightdef/.style={thick, red!70!black, <->},
    annot/.style={align=center, font=\small}
]

% Geometry
\def\xmax{10}
\def\peakx{5.6}
\def\peaky{4.3}
\def\ux{2.8}
\def\vx{8.2}

% Crossing window
\fill[violet!10] (\ux,0) rectangle (\vx,\peaky);

% Axes
\draw[axis] (-0.25,0) -- (\xmax+0.35,0)
    node[below right] {$x$};
\draw[axis] (0,-0.25) -- (0,\peaky+0.55)
    node[above left] {$T_\omega(x)$};

% Tent map
\draw[tent] (0,0) -- (\peakx,\peaky) -- (\xmax,0);

% Peak guides
\draw[guide] (\peakx,0) -- (\peakx,\peaky);
\draw[guide] (0,\peaky) -- (\peakx,\peaky);

% Peak
\fill (\peakx,\peaky) circle (2.2pt);
\node[above=3pt] at (\peakx,\peaky)
    {$\bigl(c(\omega),h(\omega)\bigr)$};

% Axis labels
\node[below left] at (0,0) {$0$};
\node[below] at (\xmax,0) {$1$};
\node[left] at (0,\peaky) {$h(\omega)$};

% x-axis markers
\draw[guide] (\ux,0) -- (\ux,{0.32*\peaky});
\draw[guide] (\vx,0) -- (\vx,{0.32*\peaky});

\node[below=3pt] at (\ux,0) {$u$};
\node[below=3pt] at (\peakx,0) {$c(\omega)$};
\node[below=3pt] at (\vx,0) {$v$};

% Branch image interval
\draw[interval] (\ux,0) -- (\vx,0);
\node[below=18pt, green!45!black] at ({(\ux+\vx)/2},0)
    {$J=[u,v]$};

% Endpoint defects
\draw[leftdef] (0,-0.62) -- (\ux,-0.62);
\node[below=3pt, blue!70!black] at ({\ux/2},-0.62)
    {$L(J)=u$};

\draw[rightdef] (\vx,-0.62) -- (\xmax,-0.62);
\node[below=3pt, red!70!black] at ({(\vx+\xmax)/2},-0.62)
    {$R(J)=1-v$};

% Slope labels
\node[annot] at (2.15,2.55) {$\text{slope }a$};
\node[annot] at (8.0,2.55) {$\text{slope }-b$};

% Crossing annotation
\node[annot, text=violet!70!black] at ({(\ux+\vx)/2},2.15)
    {$u<c(\omega)<v$\\[2pt]
     $J$ crosses the peak};

\node[annot, text=violet!70!black] at ({(\ux+\vx)/2},1.25)
    {the branch is guaranteed\\
     to split at the next layer};

\end{tikzpicture}
\caption{Endpoint defects and guaranteed splitting. A branch with image \(J=[u,v]\) has left defect \(L(J)=u\) and right defect \(R(J)=1-v\). If \(u<c(\omega)<v\), then \(J\) crosses the peak of the next tent map \(T_\omega\), so the branch splits into two child branches.}
\label{fig:defects-splitting}
\end{figure}

The observable studied in this paper is the number \(N_n\) of affine pieces of \(F_n\) on \([0,1]\), where \(F_n:=T_{\omega_n}\circ\cdots\circ T_{\omega_1}\). Equivalently, if \(A(f;J)\) denotes the smallest number of subintervals partitioning \(J\) on each of which \(f\) is affine, then
\[
N_n:=A(F_n;[0,1]).
\]
Since each layer has at most two affine pieces on \([0,1]\), one always has \(1\le N_n\le2^n\). In the deterministic Telgarsky case \(\omega_n\equiv(2,2)\), one has \(N_n=2^n\).
The model is a scalar width-three ReLU network. Indeed, for \(c=b/(a+b)\),
\[
T_\omega(x)
=
a\,\operatorname{ReLU}(x)
-(a+b)\operatorname{ReLU}(x-c)
+b\,\operatorname{ReLU}(x-1).
\]
Thus, the random tent model is directly embedded in the standard ReLU formalism.

\subsection{Large deviation upper bounds}\label{subsec:ldp-upper}

We first state the upper-bound results. These are global and use only submultiplicativity of the true branch count. 

\subsubsection{Pressure bounds for the true branch count}

The first result concerns only the true branch-count observable \(N_n\). Let \(X_n:=n^{-1}\log N_n\).

\begin{theorem}[Pressure and exponential upper bounds]\label{thm:main-pressure}
For every \(q\in\mathbb R\), the limit \[\Lambda(q):=\lim_{n\to\infty} n^{-1}\log\mathbb E[N_n^q]\]
exists. For \(q\ge0\), one has \(\Lambda(q)=\inf_{n\ge1} n^{-1}\log\mathbb E[N_n^q]\), while for \(q<0\), one has \(\Lambda(q)=\sup_{n\ge1} n^{-1}\log\mathbb E[N_n^q]\). 

	Moreover, \(\Lambda\) is finite and convex on \(\mathbb R\), and \(\min\{0,q\log 2\}\le \Lambda(q)\le \max\{0,q\log 2\}\). If \(I_+(x):=\sup_{q\ge0}\{qx-\Lambda(q)\}\) and \(I_-(x):=\sup_{q\le0}\{qx-\Lambda(q)\}\), then, for every \(x\in\mathbb R\),
\[
\limsup_{n\to\infty} n^{-1}\log\mathbb P(X_n\ge x)\le -I_+(x)
\qquad\text{and}\qquad
\limsup_{n\to\infty} n^{-1}\log\mathbb P(X_n\le x)\le -I_-(x).
\]
\end{theorem}

The proof uses only the pathwise submultiplicativity, Fekete's lemma, and exponential Markov bounds. Thus, no coarse graining, favorable subsystem, or choice of \(m\) is involved in Theorem~\ref{thm:main-pressure}. The result gives exponential upper bounds for both tails of the true observable \(N_n\), but it does not give matching lower bounds.
The argument is more robust than the one-dimensional tent model suggests. Section~\ref{sec:pressure} first isolates an abstract pressure principle for submultiplicative complexity observables. Theorem~\ref{thm:main-pressure} is then obtained as its first application. Two further consequences are recorded next.

\subsubsection{Robust extensions of the pressure argument}

The first extension concerns higher-dimensional compositional models. Let \(K\subseteq\mathbb R^d\) be a compact convex polytope. We use the convention that a map is affine on a subset of \(K\) if it agrees there with the restriction of an affine map on \(\mathbb R^d\). For a map \(f:K\to K\), let \(R_{\mathrm{cov}}(f;K)\) denote the smallest integer \(M\ge1\) for which there exist compact convex polytopes \(P_1,\dots,P_M\subseteq K\) whose union is \(K\) and such that \(f\) is affine on each \(P_i\). This is a convex-polytopal affine-cover count. It is the higher-dimensional analogue of the one-dimensional affine-piece count that is most directly compatible with composition. In the next two results, piecewise affine means that such a finite convex-polytopal affine cover exists.

The uniform bound below is an architecture-level assumption: in standard
piecewise-affine classes, such as ReLU networks with fixed widths and depth,
the constant \(B\) can be bounded explicitly in terms of the architecture.

\begin{theorem}[Higher-dimensional affine-cover pressure]\label{thm:main-polyhedral-pressure}
Let \(K\subseteq\mathbb R^d\) be a compact convex polytope, and let \((\Phi_{\omega_n})_{n\ge1}\) be i.i.d. random piecewise affine self-maps of \(K\). Set \(\Phi_n:=\Phi_{\omega_n}\circ\cdots\circ\Phi_{\omega_1}\) and \(R_n:=R_{\mathrm{cov}}(\Phi_n;K)\). Assume that \(R_n\) is measurable for every \(n\ge1\), and that 
\[\sup_{\omega}R_{\mathrm{cov}}(\Phi_{\omega};K)\le B\] 
for some deterministic \(B\ge1\). Then the conclusions of Theorem~\ref{thm:main-pressure} remain valid after replacing \(N_n\), \(X_n\), \(\Lambda\), \(I_\pm\), and \(\log 2\) by \(R_n\), \(Y_n:=n^{-1}\log R_n\), \(\Lambda_{\mathrm{cov}}\), \(I_{\mathrm{cov},\pm}\), and \(\log B\), respectively, where
\[
\Lambda_{\mathrm{cov}}(q):=\lim_{n\to\infty} n^{-1}\log\mathbb E[R_n^q]
\]
and \(I_{\mathrm{cov},\pm}\) are defined from \(\Lambda_{\mathrm{cov}}\) exactly as in Theorem~\ref{thm:main-pressure}.
\end{theorem}

\begin{example}[A computable affine-cover bound]\label{ex:relu-cover-bound}
Let \(K\subseteq\mathbb R^d\) be a compact convex polytope. Fix a ReLU
architecture with \(L\) hidden layers of widths \(m_1,\dots,m_L\), input
dimension \(d\), and output dimension \(d\). Suppose that each random map
\(\Phi_\omega:K\to K\) is the restriction to \(K\) of a network with this
architecture. Then the assumption in Theorem~\ref{thm:main-polyhedral-pressure}
holds with
\[
B
=
\prod_{\ell=1}^L H_d(m_\ell),
\qquad
H_d(m):=\sum_{j=0}^d {m\choose j}.
\]
Indeed, on each affine cell created before layer \(\ell\), the \(m_\ell\)
preactivation functions of layer \(\ell\) are affine functions of the original
input. Their zero sets cut the cell into at most \(H_d(m_\ell)\) convex
polytopes. Iterating this over the layers gives
\(
R_{\mathrm{cov}}\big(\Phi_\omega;K\big)
\le
\prod_{\ell=1}^L H_d(m_\ell)
\)
uniformly in the weights and biases. Thus, the constant \(B\) is deterministic
and computable from the architecture alone. In the one-dimensional case,
\(H_1(m)=m+1\); in particular, a single one-kink layer gives \(B=2\), matching
the two-piece tent-map case.
\end{example}

The second extension concerns one-dimensional slices through higher-dimensional compositions. Let \(\mathcal S(K)\) denote the family of affine maps \(\gamma:[0,1]\to K\), including degenerate affine segments. For a map \(f:K\to K\), define its worst-line affine-piece count by \(L(f;K):=\sup_{\gamma\in\mathcal S(K)}A(f\circ\gamma;[0,1])\).

\begin{theorem}[Worst-line pressure in higher dimension]\label{thm:main-line-pressure}
Let \(K\subseteq\mathbb R^d\) be a compact convex polytope, and let \((\Phi_{\omega_n})_{n\ge1}\) be i.i.d. random piecewise affine self-maps of \(K\). Set \(\Phi_n:=\Phi_{\omega_n}\circ\cdots\circ\Phi_{\omega_1}\) and \(L_n:=L(\Phi_n;K)\). Assume that \(L_n\) is measurable for every \(n\ge1\), and that 
\[\sup_{\omega}L(\Phi_{\omega};K)\le B\] 
for some deterministic \(B\ge1\). Then the conclusions of Theorem~\ref{thm:main-pressure} remain valid after replacing \(N_n\), \(X_n\), \(\Lambda\), \(I_\pm\), and \(\log 2\) by \(L_n\), \(Z_n:=n^{-1}\log L_n\), \(\Lambda_{\mathrm{line}}\), \(I_{\mathrm{line},\pm}\), and \(\log B\), respectively, where
\[
\Lambda_{\mathrm{line}}(q):=\lim_{n\to\infty} n^{-1}\log\mathbb E[L_n^q]
\]
and \(I_{\mathrm{line},\pm}\) are defined from \(\Lambda_{\mathrm{line}}\) exactly as in Theorem~\ref{thm:main-pressure}.
\end{theorem}

The supremum over all affine segments in Theorem~\ref{thm:main-line-pressure} is important. For a fixed deterministic segment \(\gamma\), the image \(\Phi_n\circ\gamma\) is generally a random piecewise affine curve rather than the original segment \(\gamma\). Consequently, one obtains a comparison with the worst-line complexity of the future block, not a closed submultiplicative inequality for the fixed-line observable itself.

\subsection{Large deviation lower bounds}\label{subsec:ldp-lower}

The pressure argument gives exponential upper bounds for the tails of the true observable \(N_n\). 
The lower-bound problem is more delicate. There is no corresponding supermultiplicative inequality for the true branch count, and therefore no direct lower-bound analogue of the pressure argument.

\subsubsection{Geometric obstruction and endpoint defects}

The geometric obstruction is that branch growth is no longer automatic once the layers are random. In the deterministic map, every affine branch crosses the peak of the next layer and splits into two descendants. Under perturbations, this can fail as illustrated in Figure \ref{fig:endpoint-defects-certified-splitting}. If the image interval of a branch lies too close to one endpoint of \([0,1]\), it may miss the next peak and produce only one descendant. In this sense, one side of the interval has lost its safety margin, and the branch has lost some of its future splitting potential.

Thus, the total number of branches alone does not contain enough information to propagate lower bounds. One must also record how close branch images are to losing the ability to split. For a branch whose image is \(J=[u,v]\), the relevant quantities are the endpoint defects
\[
L(J)=u,
\qquad
R(J)=1-v.
\]
Small endpoint defects guarantee that the interval crosses the peak of the next layer and hence must split. The lower-bound construction below keeps only a coarse record of these defects. 

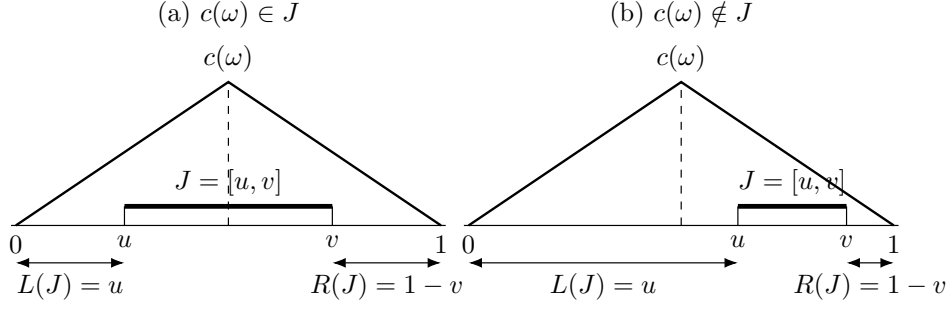
\begin{figure}[t]
\centering
\begin{tikzpicture}[x=1.25cm,y=1.15cm,>=Latex,font=\small]
\tikzset{
  axis/.style={black,line width=.45pt},
  tent/.style={black,line width=.9pt},
  guide/.style={black,dashed,line width=.45pt},
  interval/.style={black,line width=1.8pt},
  defect/.style={<->,black,line width=.45pt}
}

\def\endpointpanel#1#2#3#4{%
\begin{scope}[xshift=#1]
  \node at (2.25,2.43) {#2};
  \draw[axis] (-.05,0) -- (4.55,0);
  \draw[tent] (0,0) -- (2.25,1.65) -- (4.5,0);
  \draw[guide] (2.25,0) -- (2.25,1.65);
  \node[above] at (2.25,1.65) {$c(\omega)$};
  \node[below] at (0,0) {$0$};
  \node[below] at (4.5,0) {$1$};

  \draw[interval] (#3,.22) -- (#4,.22);
  \draw[axis] (#3,0) -- (#3,.22);
  \draw[axis] (#4,0) -- (#4,.22);
  \node[above] at ({(#3+#4)/2},.22) {$J=[u,v]$};
  \node[below] at (#3,0) {$u$};
  \node[below] at (#4,0) {$v$};

  \draw[defect] (0,-.43) -- (#3,-.43);
  \node[below] at ({#3/2},-.43) {$L(J)=u$};
  \draw[defect] (#4,-.43) -- (4.5,-.43);
  \node[below] at ({(#4+4.5)/2},-.43) {$R(J)=1-v$};
\end{scope}%
}

\endpointpanel{0cm}{(a) $c(\omega)\in J$}{1.15}{3.35}
\endpointpanel{5.99cm}{(b) $c(\omega)\notin J$}{2.85}{4.00}
\end{tikzpicture}
\caption{ In panel (a), the image interval crosses the peak of the next layer, so a split is certified. In panel (b), the image interval misses the peak, so the certified lower process does not record a split.}
\label{fig:endpoint-defects-certified-splitting}
\end{figure}

\subsubsection{Coarse graining }\label{subsubsec:coarse-graining-m}

We now turn the endpoint-defect observation into a finite-state lower process. Set \(\delta_\varepsilon:=\frac12\inf_{\omega\in U_\varepsilon}\min\{c(\omega),1-c(\omega)\}\), \(K_\varepsilon:=2+\varepsilon\), and \(\eta(\omega):=1-h(\omega)\). Since \(a,b\in[2-\varepsilon,2+\varepsilon]\), one has \(\delta_\varepsilon=(2-\varepsilon)/8\). For a fixed integer \(m\ge1\), define
\[
r_i:=\delta_\varepsilon K_\varepsilon^{-i},
\qquad i=0,\dots,m.
\]
A layer is called \emph{\(m\)-favorable} if \(\eta(\omega)\le r_m\).

The thresholds \(r_i\) form a finite defect scale. A larger value of \(i\) means a smaller certified defect and therefore a longer remaining defect budget. One step of the tent recursion can multiply an old endpoint defect by at most \(K_\varepsilon\), so a defect controlled by \(r_i\) is transformed into one controlled by \(r_{i-1}\). Increasing \(m\) improves the certified exponent available from the favorable chain, but it also makes the strongest favorable condition \(\eta(\omega)\le r_m\) more restrictive.

For fixed \(m\), define the coarse alphabet \(\mathcal A_m:=\{0,1,\dots,m,\partial\}\) and the coarse symbol \(\kappa_m(\omega)\) by \(\kappa_m(\omega):=\max\{\ell\in\{0,\dots,m\}:\eta(\omega)\le r_\ell\}\) if \(\eta(\omega)\le r_0\), and \(\kappa_m(\omega):=\partial\) otherwise. Thus, \(\kappa_m(\omega)=m\) is exactly the event that the layer is \(m\)-favorable. We write \(\pi=\pi_{m,\varepsilon}\) for the law of \(\kappa_m(\omega_1)\), and \(p_m:=\pi(m)=\mathbb P(\kappa_m(\omega_1)=m)\). The distribution of the layer does not choose \(m\), but it determines whether a given choice of \(m\) is useful. In particular, if \(p_m=0\), then the all-favorable lower-bound mechanism at level \(m\) is unavailable.
A certified token in state \((i,j)\) attached to a branch with image \(J\) means that
\(
L(J)\le r_i,
\)
and
\(
R(J)\le r_j.
\)
The indices \(i\) and \(j\) are the remaining defect budgets for the left and right endpoints. This state information is not meant to be an exact description of all branches; it is only the certificate needed to propagate a lower count.

For fixed \(m\), the favorable subsystem has states \(s_i=(i,m)\), \(i=0,\dots,m-1\). These are the states whose right endpoint has the strongest certificate. On this chain, one \(m\)-favorable layer acts by the \(m\times m\) matrix \(A_m\) with first row consisting of ones, subdiagonal entries equal to one, and all other entries equal to zero. Let \(\rho_m\) be its Perron root, that is, its spectral radius.

\begin{example}[The Fibonacci mechanism at level \(m=2\)]\label{ex:fibonacci-level-two}
At defect budget \(m=2\), the favorable chain has two states,
\[
s_1=(1,2),
\qquad
s_0=(0,2).
\]
Under a \(2\)-favorable layer, a certified branch in state \(s_1\) produces two certified favorable descendants: one in state \(s_1\) and one in state \(s_0\). A certified branch in state \(s_0\) produces only one certified favorable descendant, in state \(s_1\). Thus, in the ordered basis \((s_1,s_0)\), the certified population evolves according to
\[
A_2=
\begin{pmatrix}
1&1\\
1&0
\end{pmatrix}.
\]
Starting with one top-state token, the total certified populations are \(1,2,3,5,\dots\), and the exponential growth rate is the Perron root \(\rho_2=(1+\sqrt{5})/2\). The non-integer rate \(\rho_2\) is not a fractional one-step branching factor. Each certified branch still has either one or two certified descendants. The Perron root is the asymptotic growth rate of the deterministic mixture of these two branch types.
\end{example}

We now provide two results illustrating the use of the favorable concept to obtain lower bounds.

\begin{theorem}[Perron growth of the favorable chain]\label{thm:main-perron}
	One has \(\rho_1=1\), while \(\rho_m\in(1,2)\) for every \(m\ge2\). Moreover, \(2-\rho_m=\rho_m^{-m}\), \(\rho_m\uparrow2\), \(2-\rho_m\sim2^{-m}\), and \(\log 2-\log\rho_m\sim2^{-m-1}\). If the first \(n\) layers are all \(m\)-favorable, then \(N_n\rho_m^{-n}>0\).
\end{theorem}

\begin{theorem}[All-favorable upper-tail lower bound]\label{thm:main-all-favorable}
If \(p_m>0\), then, for every \(\delta>0\), one has \(\liminf_{n\to\infty} n^{-1}\log\mathbb P(n^{-1}\log N_n\ge\log\rho_m-\delta)\ge \log p_m\). Moreover, \(\mathbb E[N_n]\ge c_m(p_m\rho_m)^n\) for all \(n\ge1\). In particular, if \(p_m\rho_m>1\), then \(\mathbb E[N_n]\) grows exponentially.
\end{theorem}
Theorem~\ref{thm:main-perron} explains the role of large \(m\): the deterministic certified exponent \(\log\rho_m\) can be made arbitrarily close to \(\log 2\). But this does not mean that \(m\) should always be taken large. For a fixed law \(\mu_\varepsilon\), the probability \(p_m=\mathbb P(\eta(\omega_1)\le r_m)\) may decrease with \(m\), or even vanish.

Theorem \ref{thm:main-all-favorable} is a statement about the true branch count \(N_n\). Its proof simply forces the event that all layers are \(m\)-favorable and then uses the deterministic Perron growth from Theorem~\ref{thm:main-perron}. It should not be read as favoring large \(m\) in general. For a prescribed target exponent \(x<\log 2\), the all-favorable bound is most naturally invoked with the smallest \(m\) for which \(\log\rho_m>x\), up to the arbitrarily small slack allowed by \(\delta\), because larger \(m\) typically lowers \(p_m\) and therefore worsens the forcing cost \(\log p_m\). Likewise, the expectation bound is controlled by the product \(p_m\rho_m\), whose maximizing value of \(m\) need not be large.

\subsubsection{Bridged lower bounds for the upper tail}\label{subsubsec:bridged-lower-main}

The all-favorable event is useful but restrictive, because it requires every layer to have the maximally favorable symbol \(m\). 
We now provide a different construction that weakens this requirement. Instead of forcing all layers to be favorable, we divide the word into long free blocks and short single-symbol slots that we henceforth refer to as  \emph{bridges}. The free blocks are not required to be all-favorable; they may contain arbitrary coarse symbols, provided that they are productive for the certified lower transfer process introduced below. Only the bridge slots are forced to be maximally favorable.

We now make this construction more precise. Each symbol \(a\in\mathcal A_m\) determines its own linear update map \(M(a)\), and the update along a word is obtained by composing these symbol-dependent maps.  
Let
\[
\Sigma_m:=\{(i,j):0\le i,j\le m\}.
\]
We work on the vector space with basis \(\{e_{(i,j)}:(i,j)\in\Sigma_m\}\). A
basis vector \(e_{(i,j)}\) represents one certified token whose two remaining
endpoint-defect budgets are \(i\) and \(j\). The purpose of the auxiliary
transfer process is to keep only those descendants whose certificates can be verified by
the coarse symbols.
For \(a\in\mathcal A_m\), define a linear map \(M(a)\) on this vector space as
follows. If \(a\in\{0,\dots,m\}\), set
\[
M(a)e_{(i,j)}
:=
\mathbf 1_{\{i\ge1\}}e_{(i-1,a)}
+
\mathbf 1_{\{j\ge1\}}e_{(j-1,a)}.
\]
If \(a=\partial\), set \(M(\partial)e_{(i,j)}:=0\). The rule should be read as
a bookkeeping rule for certified continuations. A token in state \(e_{(i,j)}\)
may be continued from any endpoint whose budget is still positive. Continuing
from the endpoint with budget \(i\) consumes one unit of that budget and produces
\(e_{(i-1,a)}\). Continuing from the endpoint with budget \(j\) consumes one
unit of that budget and produces \(e_{(j-1,a)}\). In both cases, the new state is
oriented so that the remaining budget of the endpoint just used appears in the
first coordinate, while the newly read symbol \(a\) becomes the refreshed second
coordinate. If \(a=\partial\), no certified continuation is allowed, and all
mass is killed.
For example, if \(m\ge3\), then
\(
M(3)e_{(2,1)}=e_{(1,3)}+e_{(0,3)}.
\)
The two summands correspond to the two possible endpoint continuations. In
contrast,
\(
M(1)e_{(0,2)}=e_{(1,1)},
\)
because the endpoint with budget \(0\) cannot be used. Thus, \(M(a)\) records
all certified continuations through the symbol \(a\), while discarding
uncertified descendants.

The favorable chain is the subspace
\[
F_m
:=
\operatorname{span}\{e_{s_i}:s_i=(i,m),\ i=0,\dots,m-1\}.
\]
These are the auxiliary states whose second coordinate is the strongest
certificate value \(m\). In this sense, \(F_m\) is the part of the auxiliary
state space that remains well positioned for later bridging.
Let \(P_F\) denote the coordinate projection onto \(F_m\). Thus, \(P_F\) keeps
the coordinates belonging to the favorable chain and kills all other
coordinates. Let \(\mathbf 1_F\) be the vector on \(F_m\) whose coordinates in
the basis \(\{e_{s_i}:0\le i\le m-1\}\) are all equal to one. Consequently, for
any nonnegative vector \(x\), the scalar \(\mathbf 1_F^\top P_Fx\) is the total
mass of \(x\) lying on the favorable chain.
Finally, let
\(
e_\star:=e_{s_{m-1}}=e_{(m-1,m)}
\)
be the top favorable-chain basis vector. For a word
\(w=(a_1,\dots,a_\ell)\), read from left to right, write
\[
M(w):=M(a_\ell)\cdots M(a_1).
\]
We define the certified favorable block score by
\[
\Gamma(w):=\mathbf 1_F^\top P_FM(w)e_\star.
\]
In words, \(\Gamma(w)\) is the total mass of certified tokens that lie on the
favorable chain after the block \(w\) has been read, starting from one token in state \(e_\star\). It is not the total number of certified
tokens produced by the block. It counts only those tokens ending in
favorable-chain states and thus available for the subsequent bridge
construction.
The bridge word is the one-letter word
\[
b^{(m)}:=(m)\in\mathcal A_m.
\]
Its role is to return  to the top favorable state
\(e_\star=e_{(m-1,m)}\), from which the next free block can be started.
Indeed, every favorable-chain state \(s_i=(i,m)\), \(0\le i\le m-1\), produces at least one copy of \(e_\star\) after reading the single favorable symbol \(m\). Consequently, for all words \(u,v\),
\[
\Gamma\big(vb^{(m)}u\big)\ge \Gamma(v)\Gamma(u).
\]
Thus, inserting one maximally favorable layer between two free blocks makes the certified favorable score supermultiplicative.

\begin{example}
It may be useful to see this bridge in the smallest nontrivial case. When
\(m=2\), the favorable chain is
\[
F_2=\operatorname{span}\{e_{s_1},e_{s_0}\},
\qquad
s_1=(1,2),\quad s_0=(0,2),
\]
and \(e_\star=e_{s_1}\). Reading the favorable symbol \(2\) gives
\[
M(2)e_{s_1}=e_{s_1}+e_{s_0},
\qquad
M(2)e_{s_0}=e_{s_1}.
\]
Thus, in the ordered basis \((e_{s_1},e_{s_0})\), the favorable-chain transition matrix is
\[
A_2=
\begin{pmatrix}
1&1\\
1&0
\end{pmatrix}.
\]
\end{example}

Suppose that the favorable part after reading a word \(u\) is
\(
P_FM(u)e_\star=x_1e_{s_1}+x_0e_{s_0}.
\)
Then, \(\Gamma(u)=x_1+x_0\). After one bridge symbol \(2\), the coefficient of the top state \(e_\star=e_{s_1}\) is exactly \(x_1+x_0\). Each such copy of \(e_\star\) can be used as a fresh starting token for the next word \(v\), contributing \(\Gamma(v)\) favorable tokens. This gives
\[
\Gamma\big(vb^{(2)}u\big)\ge \Gamma(v)\Gamma(u).
\]
The general case is the same restart mechanism with \(m\) possible budget levels instead of two.

The reason the bridge helps is deterministic. The certified score \(\Gamma(w)\) is not multiplicative in \(w\): after a block \(u\), the surviving certified tokens may end in different favorable-chain states. The one-symbol bridge \(b^{(m)}=(m)\) refreshes every such favorable token by creating a copy of the common restart state \(e_\star\). Thus, the bridge converts many individually productive free blocks into one productive long word. Probabilistically, the cost is paid only for the bridge slots, not for every layer. If the free blocks have length \(r\), the forced favorable fraction is roughly \(1/(r+1)\), and this overhead becomes small when \(r\) is large. The lower bound in Theorem~\ref{thm:main-bridge-lower} below is exactly this tradeoff: a relative-entropy cost for choosing productive free blocks, plus the explicit cost of inserting favorable bridges.

Fix a free-block length \(r\ge1\). Let \(\mathcal W_r:=\mathcal A_m^r\), \(\pi_r:=\pi^{\otimes r}\), \(\mathcal G_r:=\{w\in\mathcal W_r:\Gamma(w)>0\}\), and \(g_r(w):=\log\Gamma(w)\) for \(w\in\mathcal G_r\). The probability of one bridge word is \(p_b:=\pi(b^{(m)})=p_m\). Assuming \(p_b>0\), define the bridge construction rate \(J_{r,m}^{+}(x)\) as the infimum of \((H(\nu\mid\pi_r)-\log p_b)/(r+1)\) over all \(\nu\in\mathcal P(\mathcal W_r)\) such that \(\nu(\mathcal G_r)=1\) and \(\int g_r\,d\nu>(r+1)x\). Here \(\mathcal P(\mathcal W_r)\) denotes the probability simplex on \(\mathcal W_r\), and \(H(\nu\mid\pi_r)\) denotes relative entropy. The infimum over the empty set is \(+\infty\).
The bridge construction is a sufficient lower-bound mechanism. It counts only certified branches and does not characterize all branch configurations that can lead to large values of \(N_n\).

\begin{theorem}[Bridged lower bound for the upper tail]\label{thm:main-bridge-lower}
Fix \(m,r\ge1\), and assume \(p_b>0\). Let \(n_k:=1+kr+(k-1)\). Then, for every \(x\in\mathbb R\), one has 
\[\liminf_{k\to\infty} n_k^{-1}\log\mathbb P(n_k^{-1}\log N_{n_k}>x)\ge -J_{r,m}^{+}(x).\]
\end{theorem}

The initial \(+1\) in \(n_k=1+kr+(k-1)\) is only a seed layer. On the event that the first coarse symbol is \(m\), the initial branch produces certified tokens in the top favorable state \(e_\star\). This allows the later block construction to start from the vector used in \(\Gamma\). The seed cost is negligible on the exponential scale.

The parameters \(m\) and \(r\) have different roles. The parameter \(m\) sets the defect resolution and the favorable symbol used for bridging. The parameter \(r\) is a blocking parameter for the empirical law of the free blocks. For fixed \(m\), larger \(r\) reduces the bridge overhead per unit depth, but it does not remove the possible gap between the certified count \(\Gamma\) and the true branch count \(N_n\).

\begin{remark}[Regenerative interpretation]\label{rem:regenerative-interpretation}
The bridge construction has a regenerative interpretation. Once certified
mass lies in the favorable chain, a single favorable symbol \(m\) creates at
least one copy of the common restart state \(e_\star=e_{(m-1,m)}\) from every
favorable-chain state. Thus, favorable symbols serve as restart symbols
for the certified transfer process.
One can make this analogy literal by cutting the coarse-symbol sequence at
non-overlapping first occurrences of the word \(mm\): after each
cut, one waits until the first adjacent occurrence of \(mm\), consumes that
terminal word, and then cuts again. The resulting cycles are i.i.d., and each
cycle carries a length and a certified reward, namely the logarithm of the
number of regenerated \(e_\star\)-tokens produced by the corresponding matrix
product. This leads to a renewal-reward version of the same lower-bound
mechanism, with a variational cost of the form
\[
\inf_\nu
\frac{H(\nu\mid Q)}{\int s\,d\nu},
\]
where \(Q\) is the law of one regeneration cycle, \(s\) is its length, and the
infimum is restricted to cycle laws whose average certified reward exceeds the
desired linear growth rate. This viewpoint is analogous to large-deviation
results for regenerative and renewal-reward processes
\cite{KuczekCrank1991,GlynnWhitt1994,Zamparo2023}.
We do not use the regenerative formulation below. The fixed-block statement in
Theorem~\ref{thm:main-bridge-lower} is deliberately more elementary: it works on
the finite alphabet \(\mathcal W_r\), uses only the method of types, and keeps
the bridge cost explicit. The regenerative viewpoint is nevertheless a useful
way to understand why favorable symbols restore concatenability.
\end{remark}

\subsubsection{Local lower-tail exclusion near \(\log 2\)}

Finally, for fixed \(m\), the following condition says that the random environment is so close to the Telgarsky map that every layer is favorable at level \(m\). We say that the model is uniformly favorable at level \(m\) if \(\eta_\ast(\varepsilon):=\sup_{\omega\in U_\varepsilon}\eta(\omega)\le r_m\). In this case, every layer is \(m\)-favorable almost surely.
The deterministic mechanism behind the next theorem is the favorable-chain growth illustrated in Example~\ref{ex:fibonacci-level-two}. At general level \(m\), the same mechanism has a longer defect budget and Perron rate \(\rho_m\uparrow2\).

The order of choices is important. For a prescribed gap \(\xi\), one first chooses \(m\) so that \(\log\rho_m\) is close enough to \(\log 2\). Only after \(m\) is fixed does one choose the noise level \(\varepsilon\) small enough that all admissible layers are \(m\)-favorable.

\begin{theorem}[Local lower-tail exclusion near \(\log 2\)]\label{thm:main-local}
For every \(\xi>0\), there exist \(m(\xi)\ge1\), \(\varepsilon_\ast(\xi)>0\), and \(n_0(\xi)\ge1\) such that, whenever \(0<\varepsilon\le\varepsilon_\ast(\xi)\) and \(n\ge n_0(\xi)\), one has \[\mathbb P(n^{-1}\log N_n\le\log 2-\xi)=0.\]
\end{theorem}

\section{Proofs of Theorems \ref{thm:main-pressure}, \ref{thm:main-polyhedral-pressure} and \ref{thm:main-line-pressure} --Subadditive pressure and exponential upper bounds}\label{sec:pressure}

This section proves Theorem~\ref{thm:main-pressure} and its higher-dimensional extensions, Theorems~\ref{thm:main-polyhedral-pressure} and \ref{thm:main-line-pressure}. The argument is organized in three steps. First, we isolate an abstract pressure theorem for any positive submultiplicative complexity observable with independent increments. Second, we verify its hypotheses for the one-dimensional branch-count observable \(N_n\). Third, we verify the same hypotheses for higher-dimensional affine-cover counts and worst-line affine-piece counts.

First, we record the elementary facts about the layer maps and their branch-count observable.

\subsection{Layer maps}

The first lemma identifies the peak and image of a single random tent layer.

\begin{lemma}\label{lem:tent-basic}
For every \(\omega=(a,b)\in U_\varepsilon\), the map \(T_\omega\) is continuous and piecewise affine with unique peak location \(c(\omega)=b/(a+b)\) and peak height \(h(\omega)=ab/(a+b)\le1\). In particular,
\[
T_\omega([0,1])=[0,h(\omega)]\subseteq [0,1].
\]
\end{lemma}

\begin{proof}
The affine functions \(x\mapsto ax\) and \(x\mapsto b(1-x)\) intersect at \(x=c(\omega)\), and their common value there is \(h(\omega)\). Since \(T_\omega\) is their minimum clipped below by \(0\), it is continuous and piecewise affine. The condition \(h(\omega)\le1\) gives \(T_\omega([0,1])\subseteq[0,1]\).
\end{proof}

The next proposition embeds every tent layer into the scalar ReLU formalism.
Thus, as mentioned before.
\begin{proposition}[ReLU representation]\label{prop:relu-representation}
Let \(\operatorname{ReLU}(t):=\max\{t,0\}\). For \(\omega=(a,b)\in U_\varepsilon\), with \(c=b/(a+b)\), one has
\[
T_\omega(x)=a\operatorname{ReLU}(x)-(a+b)\operatorname{ReLU}(x-c)+b\operatorname{ReLU}(x-1).
\]
\end{proposition}

\begin{proof}
The right-hand side vanishes for \(x<0\). For \(0\le x\le c\), only the first ReLU is active, and the value is \(ax\). For \(c\le x\le1\), the first two ReLUs are active, and the value is \(ax-(a+b)(x-c)=b(1-x)\). For \(x\ge1\), all three ReLUs are active, and the value is \(ax-(a+b)(x-c)+b(x-1)=0\). This agrees with \(T_\omega\) on all four intervals.
\end{proof}

\subsection{An abstract submultiplicative pressure principle}

The next proposition is the common mechanism behind all our pressure upper bounds.

\begin{proposition}[Submultiplicative pressure principle]\label{prop:abstract-pressure}
Let \((C_n)_{n\ge1}\) be random variables such that
\begin{enumerate}
	\item \(1\le C_n\le B^n\) almost surely for a deterministic constant \(B\ge1\), and
	\item for every \(n,\ell\ge1\), there exists a random variable \(C_\ell'\) with the same law as \(C_\ell\), independent of \(C_n\), and satisfying \(C_{n+\ell}\le C_nC_\ell'\) almost surely. 
\end{enumerate}
	 Then for every \(q\in\mathbb R\), the limit \(\Lambda_C(q):=\lim_{n\to\infty} n^{-1}\log\mathbb E[C_n^q]\) exists. For \(q\ge0\), one has \(\Lambda_C(q)=\inf_{n\ge1} n^{-1}\log\mathbb E[C_n^q]\), while for \(q<0\), one has \(\Lambda_C(q)=\sup_{n\ge1} n^{-1}\log\mathbb E[C_n^q]\). Moreover, \(\Lambda_C\) is finite and convex on \(\mathbb R\), and \(\min\{0,q\log B\}\le\Lambda_C(q)\le\max\{0,q\log B\}\). If \(I_{C,+}(x):=\sup_{q\ge0}\{qx-\Lambda_C(q)\}\) and \(I_{C,-}(x):=\sup_{q\le0}\{qx-\Lambda_C(q)\}\), then, for every \(x\in\mathbb R\),
\[
\limsup_{n\to\infty} n^{-1}\log\mathbb P(Y_n\ge x)\le -I_{C,+}(x)
\qquad\text{and}\qquad
\limsup_{n\to\infty} n^{-1}\log\mathbb P(Y_n\le x)\le -I_{C,-}(x),
\]
where \(Y_n:=n^{-1}\log C_n\).
\end{proposition}

\begin{proof}
	For \(q\in\mathbb R\), define \(a_n(q):=\log\mathbb E[C_n^q]\). Since \(1\le C_n\le B^n\), every moment \(\mathbb E[C_n^q]\) is finite and strictly positive, so \(a_n(q)\) is well defined. We treat the cases \(q \ge 0\) and \(q < 0\) separately.
	\medskip

\noindent	\(\boldsymbol{q\ge0}\). The map \(t\mapsto t^q\) is increasing on \((0,\infty)\), hence \(C_{n+\ell}^q\le C_n^q(C_\ell')^q\). Taking expectations and using independence of \(C_n\) and \(C_\ell'\) gives \(\mathbb E[C_{n+\ell}^q]\le\mathbb E[C_n^q]\mathbb E[C_\ell^q]\). Thus, \(a_{n+\ell}(q)\le a_n(q)+a_\ell(q)\), so Fekete's lemma yields existence of \(\lim_{n\to\infty} n^{-1}a_n(q)\) together with the infimum formula.
\medskip

	\noindent\(\boldsymbol{q<0}\). The map \(t\mapsto t^q\) is decreasing on \((0,\infty)\), hence \(C_{n+\ell}^q\ge C_n^q(C_\ell')^q\). Taking expectations and using independence gives \(\mathbb E[C_{n+\ell}^q]\ge\mathbb E[C_n^q]\mathbb E[C_\ell^q]\). Thus, \(a_{n+\ell}(q)\ge a_n(q)+a_\ell(q)\), so the superadditive version of Fekete's lemma yields existence of \(\lim_{n\to\infty} n^{-1}a_n(q)\) together with the supremum formula.

The deterministic bound \(1\le C_n\le B^n\) implies \(1\le C_n^q\le B^{nq}\) for \(q\ge0\), and \(B^{nq}\le C_n^q\le1\) for \(q<0\). After taking expectations, logarithms, dividing by \(n\), and passing to the limit, one obtains \(\min\{0,q\log B\}\le\Lambda_C(q)\le\max\{0,q\log B\}\). Hence, \(\Lambda_C\) is finite on \(\mathbb R\).
For each fixed \(n\), the map \(q\mapsto n^{-1}\log\mathbb E[C_n^q]\) is convex. Indeed, it is \(n^{-1}\) times the log-Laplace transform of the bounded random variable \(\log C_n\). Since \(\Lambda_C\) is the pointwise finite limit of these convex functions, the defining convexity inequality passes to the limit, and \(\Lambda_C\) is convex.

It remains to prove the tail bounds. Since \(e^{nqY_n}=C_n^q\), exponential Markov inequalities give
\[
\mathbb P(Y_n\ge x)\le e^{-nqx}\mathbb E[C_n^q]\quad(q\ge0),
\qquad
\mathbb P(Y_n\le x)\le e^{-nqx}\mathbb E[C_n^q]\quad(q\le0).
\]
Taking logarithms, dividing by \(n\), passing to \(\limsup\), and then optimizing over \(q\ge0\) in the first inequality and \(q\le0\) in the second proves the claimed bounds.
\end{proof}

\subsection{Application to branch counts in the random tent model}

We now specialize Proposition~\ref{prop:abstract-pressure} to the true branch count \(N_n\).

The deterministic counting inequality behind the next lemma is standard for
compositions of piecewise-affine maps; compare
\cite[Sec.~6.1]{PetersenZech2024}. We include the proof because the
shifted-block variable and its independence from the past are the probabilistic
input needed for Theorem~\ref{thm:main-pressure}.

\begin{lemma}\label{lem:submultiplicative-N}
For all \(n,\ell\ge1\), if \(F_{n,\ell}:=T_{\omega_{n+\ell}}\circ\cdots\circ T_{\omega_{n+1}}\), then
\[
N_{n+\ell}\le N_nN_\ell',
\]
where \(N_\ell':=A(F_{n,\ell};[0,1])\) has the same law as \(N_\ell\) and is independent of \(N_n\).
\end{lemma}

\begin{proof}
Write \(F_{n+\ell}=H\circ G\) with \(G:=F_n\) and \(H:=F_{n,\ell}\). By definition of \(N_n=A(G;[0,1])\), there exists a partition of \([0,1]\) into \(N_n\) intervals \(J_1,\dots,J_{N_n}\) such that \(G\) is affine on every \(J_i\). Since \(G([0,1])\subseteq[0,1]\), the image \(G(J_i)\) is a compact subinterval of \([0,1]\).

By definition of \(N_\ell'=A(H;[0,1])\), the map \(H\) has an affine partition of \([0,1]\) with \(N_\ell'\) intervals. Restricting this partition to the subinterval \(G(J_i)\) shows that \(H\) has at most \(N_\ell'\) affine pieces on \(G(J_i)\). Since \(G\) is affine on \(J_i\), the composition \(H\circ G\) therefore has at most \(N_\ell'\) affine pieces on \(J_i\). Summing this bound over the \(N_n\) intervals \(J_i\) gives \(N_{n+\ell}\le N_nN_\ell'\).

The variable \(N_n\) is determined by \((\omega_1,\dots,\omega_n)\), while \(N_\ell'\) is determined by \((\omega_{n+1},\dots,\omega_{n+\ell})\). These environment blocks are independent because the layers are i.i.d. Hence, \(N_n\) and \(N_\ell'\) are independent. Since the block \((\omega_{n+1},\dots,\omega_{n+\ell})\) has the same law as \((\omega_1,\dots,\omega_\ell)\), the variable \(N_\ell'\) has the same law as \(N_\ell\).
\end{proof}

We now prove the first pressure theorem. It  gives exponential upper bounds for both tails of the true observable \(X_n=n^{-1}\log N_n\). It does not by itself prove a full large-deviation principle, because it does not provide matching lower bounds. The later bridge construction supplies one lower-bound mechanism for the upper tail, but not a complete converse to the pressure upper bound.

\begin{proof}[Proof of Theorem~\ref{thm:main-pressure}]
 Lemma~\ref{lem:submultiplicative-N} gives, for every \(n,\ell\ge1\), a shifted-block variable \(N_\ell'\) with the same law as \(N_\ell\), independent of \(N_n\), and satisfying \(N_{n+\ell}\le N_nN_\ell'\) almost surely. Thus, all hypotheses of Proposition~\ref{prop:abstract-pressure} hold with \(C_n:=N_n\) and \(B:=2\). The conclusion of Proposition~\ref{prop:abstract-pressure} is exactly the statement of Theorem~\ref{thm:main-pressure}.
\end{proof}

\subsection{Higher-dimensional affine-cover complexity}

We next prove Theorem~\ref{thm:main-polyhedral-pressure}. The main point is that convex-polytopal affine covers are exactly multiplicative under composition.

Let \(K\subseteq\mathbb R^d\) be a compact convex polytope. For a map \(f:K\to K\), recall that \(R_{\mathrm{cov}}(f;K)\) is the smallest integer \(M\ge1\) for which \(K\) can be covered by compact convex polytopes \(P_1,\dots,P_M\subseteq K\) such that \(f\) agrees on each \(P_i\) with the restriction of an affine map on \(\mathbb R^d\).

\begin{lemma}[Submultiplicativity of affine-cover counts]\label{lem:submultiplicative-cover}
For piecewise affine \(G,H:K\to K\),
\[
R_{\mathrm{cov}}(H\circ G;K)\le R_{\mathrm{cov}}(G;K)\,R_{\mathrm{cov}}(H;K).
\]
\end{lemma}

\begin{proof}
Set \(r:=R_{\mathrm{cov}}(G;K)\) and \(s:=R_{\mathrm{cov}}(H;K)\). By definition of \(R_{\mathrm{cov}}\), there exist compact convex polytopes \(P_1,\dots,P_r\subseteq K\) covering \(K\) such that \(G\) is affine on each \(P_i\), and compact convex polytopes \(Q_1,\dots,Q_s\subseteq K\) covering \(K\) such that \(H\) is affine on each \(Q_j\).
For every \(i\in\{1,\dots,r\}\), choose an affine map \(L_i:\mathbb R^d\to\mathbb R^d\) whose restriction to \(P_i\) agrees with \(G\). For every \(j\in\{1,\dots,s\}\), choose an affine map \(M_j:\mathbb R^d\to\mathbb R^d\) whose restriction to \(Q_j\) agrees with \(H\). For \(1\le i\le r\) and \(1\le j\le s\), define
\[
P_{ij}:=P_i\cap L_i^{-1}(Q_j).
\]
Each set \(P_{ij}\) is either empty or a compact convex polytope. Indeed, \(Q_j\) is a compact convex polytope, its inverse image under the affine map \(L_i\) is a closed convex polyhedron, and intersecting with the compact convex polytope \(P_i\) yields a compact convex polytope.

The nonempty sets \(P_{ij}\) cover \(K\). To see this, fix \(x\in K\). Since the \(P_i\) cover \(K\), one may choose \(i\) such that \(x\in P_i\). Since \(G(x)=L_i(x)\in K\) and the \(Q_j\) cover \(K\), one may choose \(j\) such that \(G(x)\in Q_j\). Then \(x\in P_i\cap L_i^{-1}(Q_j)=P_{ij}\).
Finally, on \(P_{ij}\), one has \(G=L_i\) and \(L_i(P_{ij})\subseteq Q_j\), so \(H\circ G=M_j\circ L_i\), which is affine. Thus, \(H\circ G\) is affine on every nonempty \(P_{ij}\). There are at most \(rs\) such sets, and therefore \(R_{\mathrm{cov}}(H\circ G;K)\le rs\).
\end{proof}

We now derive the pressure theorem for \(R_n\).

\begin{proof}[Proof of Theorem~\ref{thm:main-polyhedral-pressure}]
For \(n\ge1\), set \(\Phi_n:=\Phi_{\omega_n}\circ\cdots\circ\Phi_{\omega_1}\) and \(R_n:=R_{\mathrm{cov}}(\Phi_n;K)\). For \(n,\ell\ge1\), define the shifted composition \(\Phi_{n,\ell}:=\Phi_{\omega_{n+\ell}}\circ\cdots\circ\Phi_{\omega_{n+1}}\) and \(R_\ell':=R_{\mathrm{cov}}(\Phi_{n,\ell};K)\). Since \(\Phi_{n+\ell}=\Phi_{n,\ell}\circ\Phi_n\), Lemma~\ref{lem:submultiplicative-cover} gives \(R_{n+\ell}\le R_nR_\ell'\).
The variable \(R_n\) is determined by \((\omega_1,\dots,\omega_n)\), while \(R_\ell'\) is determined by \((\omega_{n+1},\dots,\omega_{n+\ell})\). The layer sequence is i.i.d., hence these two variables are independent, and \(R_\ell'\) has the same law as \(R_\ell\).

It remains to verify the deterministic exponential bound. By hypothesis, \(R_{\mathrm{cov}}(\Phi_{\omega_j};K)\le B\) almost surely for every \(j\). Repeated application of Lemma~\ref{lem:submultiplicative-cover} gives \(R_n\le B^n\) almost surely. Since \(R_n\ge1\), all hypotheses of Proposition~\ref{prop:abstract-pressure} hold with \(C_n:=R_n\). Its conclusion is precisely Theorem~\ref{thm:main-polyhedral-pressure}.
\end{proof}

\begin{remark}
The affine-cover count is deliberately formulated as a cover number rather than as the number of maximal affine regions. This choice eliminates boundary and degeneracy issues and makes the multiplicative inequality exact. In one dimension, the corresponding cover count coincides with the usual affine-piece count on an interval.
\end{remark}

\subsection{Worst-line affine-piece complexity}

We finally prove Theorem~\ref{thm:main-line-pressure}. This extension is close in spirit to the 1D tent observable, because the complexity is still measured by numbers of affine pieces on intervals. The only difference is that one takes the worst affine segment in the higher-dimensional state space.
Let \(K\subseteq\mathbb R^d\) be a compact convex polytope, and let \(\mathcal S(K)\) be the family of affine maps \(\gamma:[0,1]\to K\), including degenerate segments. For \(f:K\to K\), recall that \(L(f;K):=\sup_{\gamma\in\mathcal S(K)}A(f\circ\gamma;[0,1])\).

\begin{lemma}\label{lem:submultiplicative-line}
For piecewise affine maps \(G,H:K\to K\),
\[
L(H\circ G;K)\le L(G;K)\,L(H;K).
\]
\end{lemma}

\begin{proof}
Fix an affine segment \(\gamma\in\mathcal S(K)\). Let \(m:=A(G\circ\gamma;[0,1])\). By definition of \(A\), there exists a partition \(0=t_0<t_1<\cdots<t_m=1\) such that \(G\circ\gamma\) is affine on each interval \([t_{i-1},t_i]\).

Fix \(i\in\{1,\dots,m\}\). Since \(G\circ\gamma\) is affine on \([t_{i-1},t_i]\), its restriction to this interval is an affine parametrization of a line segment in \(K\), possibly degenerate. After affine reparametrization of \([t_{i-1},t_i]\) onto \([0,1]\), this restriction belongs to \(\mathcal S(K)\). Therefore, by definition of \(L(H;K)\), the composition \(H\circ G\circ\gamma\) has at most \(L(H;K)\) affine pieces on \([t_{i-1},t_i]\).

Summing over the \(m\) intervals in the chosen partition gives \(A(H\circ G\circ\gamma;[0,1])\le m\,L(H;K)\). Since \(m=A(G\circ\gamma;[0,1])\le L(G;K)\), one obtains \(A(H\circ G\circ\gamma;[0,1])\le L(G;K)L(H;K)\). Taking the supremum over \(\gamma\in\mathcal S(K)\) proves the claim.
\end{proof}

We now derive the pressure theorem for \(L_n\).

\begin{proof}[Proof of Theorem~\ref{thm:main-line-pressure}]
For \(n\ge1\), set \(\Phi_n:=\Phi_{\omega_n}\circ\cdots\circ\Phi_{\omega_1}\) and \(L_n:=L(\Phi_n;K)\). For \(n,\ell\ge1\), define the shifted composition \(\Phi_{n,\ell}:=\Phi_{\omega_{n+\ell}}\circ\cdots\circ\Phi_{\omega_{n+1}}\) and \(L_\ell':=L(\Phi_{n,\ell};K)\). Since \(\Phi_{n+\ell}=\Phi_{n,\ell}\circ\Phi_n\), Lemma~\ref{lem:submultiplicative-line} gives \(L_{n+\ell}\le L_nL_\ell'\).
The variable \(L_n\) is determined by \((\omega_1,\dots,\omega_n)\), while \(L_\ell'\) is determined by \((\omega_{n+1},\dots,\omega_{n+\ell})\). Since the layers are i.i.d., \(L_n\) and \(L_\ell'\) are independent, and \(L_\ell'\) has the same law as \(L_\ell\).
By hypothesis, \(L(\Phi_{\omega_j};K)\le B\) almost surely for every \(j\). Repeated application of Lemma~\ref{lem:submultiplicative-line} gives \(L_n\le B^n\) almost surely. Since \(L_n\ge1\), Proposition~\ref{prop:abstract-pressure} applies with \(C_n:=L_n\). Its conclusion is exactly Theorem~\ref{thm:main-line-pressure}.
\end{proof}

\begin{remark}
The same argument does not close for a fixed deterministic segment \(\gamma\). One always has a one-step comparison between the affine-piece count of \((H\circ G)\circ\gamma\) and the worst-line complexity of \(H\), but the future block is evaluated along the random piecewise affine image \(G\circ\gamma\), not along the original \(\gamma\). Thus, the natural closed observable is the supremum \(L(f;K)\), rather than the fixed-line count \(A(f\circ\gamma;[0,1])\).
\end{remark}

\section{The favorable subsystem and the bridge lemma}\label{sec:favorable}

This section turns the geometric obstruction described in Section~\ref{subsec:ldp-lower} into a finite-state lower process. A branch is useful for future growth only if its image interval remains sufficiently far from both endpoints of \([0,1]\), so that it is forced to cross the next peak. We therefore track only certified endpoint information. The resulting process is not an exact branch-count process; it is a lower transfer process counting selected branches whose future splitting can be guaranteed.

For an affine branch \(J\) of a partial composition with image \([u,v]\), recall that
\(
L(J):=u,
\)
and
\(
R(J):=1-v.
\)
The thresholds \(r_i=\delta_\varepsilon K_\varepsilon^{-i}\), \(i=0,\dots,m\), were defined in Section~\ref{subsubsec:coarse-graining-m}. We note that the number \(\delta_\varepsilon\) is a strict uniform safety margin: it is half of the minimum possible distance from the peak location \(c(\omega)\) to the endpoints \(0\) and \(1\). Hence, if a branch image \(J=[u,v]\) satisfies \(L(J)\le\delta_\varepsilon\) and \(R(J)\le\delta_\varepsilon\), then \(u<c(\omega)<v\) for every admissible next layer, and the branch is forced to split.
The geometric sequence \(r_i=\delta_\varepsilon K_\varepsilon^{-i}\) records a finite defect budget. Since one application of a tent map can multiply an old endpoint defect by at most \(K_\varepsilon=2+\varepsilon\), a defect controlled at level \(i\) is guaranteed to remain controlled at level \(i-1\) after one step. Thus, the index \(i\) measures how many future layers of worst-case expansion the corresponding endpoint can still tolerate.

We now show that the matrices \(M(a)\) defined in Section~\ref{subsubsec:bridged-lower-main} propagate certified branch tokens. A token in state \((i,j)\in\Sigma_m\) attached to a branch \(J\) means that \(
L(J)\le r_i,\)
and \(
R(J)\le r_j.\)
The indices \(i\) and \(j\) are defect budgets for the left and right endpoints. A larger index means a smaller defect and hence more certified future splitting potential.
The same branch may satisfy many such inequalities. Thus, \(\Sigma_m\) is not intended as a disjoint state partition. It is a finite set of lower-control labels. The auxiliary process uses one propagated token per selected branch; the token label records one certificate that is sufficient for the next step and need not be the strongest certificate available for that branch. 

The next proposition gives the certified transfer process for the true branch count.

\begin{proposition}\label{prop:Z-lower-bound}
Let \(a_n:=\kappa_m(\omega_n)\in\mathcal A_m\), and define \(Z_0:=e_{(m,m)}\) and \(Z_n:=M(a_n)\cdots M(a_1)e_{(m,m)}\). Then \(Z_n(i,j)\) is a lower bound on the number of depth-\(n\) branches carrying a propagated token in state \((i,j)\), and
\[
N_n\ge\mathbf 1^\top Z_n.
\]
\end{proposition}
\begin{proof}
At time \(n=0\), there is one branch, namely \([0,1]\), and its defects are
zero. Hence it carries the strongest certified state \((m,m)\), so
\(Z_0=e_{(m,m)}\) is valid.

Assume the statement holds at time \(n\). Each token counted by \(Z_n(i,j)\)
is attached to a selected depth-\(n\) branch carrying the certificate
\((i,j)\), and no selected branch carries more than one propagated token. Let
the next layer be \(T_{\omega_{n+1}}\), and write
\(q:=\kappa_m(\omega_{n+1})\).

We first verify directly the one-step transition encoded by \(M(q)\). If
\(q=\partial\), then no certified continuation is recorded by the auxiliary
process, in agreement with \(M(\partial)=0\). Suppose therefore that
\(q\in\{0,\dots,m\}\). Let a selected branch \(J\) carry a token in state
\((i,j)\), and write its image interval as \(J=[u,v]\). The certificate means
that \(L(J)\le r_i\) and \(R(J)\le r_j\). In particular,
\(L(J)\le\delta_\varepsilon\) and \(R(J)\le\delta_\varepsilon\). By the
definition of \(\delta_\varepsilon\), one has
\(\delta_\varepsilon<c(\omega_{n+1})\) and
\(\delta_\varepsilon<1-c(\omega_{n+1})\) for every admissible layer. Hence
\(u<c(\omega_{n+1})<v\), so the next layer splits \(J\) into a left child and
a right child.

Let \(a_{\omega_{n+1}}\) and \(b_{\omega_{n+1}}\) denote the left and right
slopes of \(T_{\omega_{n+1}}\). On the left of the peak the map is
\(x\mapsto a_{\omega_{n+1}}x\), while on the right of the peak it is
\(x\mapsto b_{\omega_{n+1}}(1-x)\). Thus the endpoint defects of the two
children are
\[
(L_L,R_L)=\bigl(a_{\omega_{n+1}}L(J),\eta(\omega_{n+1})\bigr),
\qquad
(L_R,R_R)=\bigl(b_{\omega_{n+1}}R(J),\eta(\omega_{n+1})\bigr).
\]
Since \(q=\kappa_m(\omega_{n+1})\), the definition of the coarse symbol gives
\(\eta(\omega_{n+1})\le r_q\). Also
\(a_{\omega_{n+1}},b_{\omega_{n+1}}\le K_\varepsilon\). Therefore, if
\(i\ge1\), then
\(L_L\le K_\varepsilon r_i=r_{i-1}\) and \(R_L\le r_q\), so the left child
carries the certified state \((i-1,q)\). Similarly, if \(j\ge1\), then
\(L_R\le K_\varepsilon r_j=r_{j-1}\) and \(R_R\le r_q\), so the right child
carries the certified state \((j-1,q)\). This is exactly the one-step rule
recorded by \(M(q)e_{(i,j)}\).
We now apply this verification to every selected token at time \(n\). The two
child tokens arising from one parent token, when both are present, are attached
to the left and right child branches of that parent and hence to distinct
branches. Tokens arising from different parent branches are also attached to
distinct child branches. Thus no depth-\(n+1\) branch receives more than one
propagated token. Summing over all parent tokens gives
\(Z_{n+1}=M(q)Z_n=M(a_{n+1})Z_n\), and the total number of propagated tokens
at depth \(n+1\) cannot exceed the number of actual depth-\(n+1\) branches.
This closes the induction and proves the claim.
\end{proof}

The favorable chain consists of the states \(s_i=(i,m)\), \(i=0,\dots,m-1\), whose right defect is controlled at the strongest threshold \(r_m\). Also, 
if the next layer is \(m\)-favorable, then the certified dynamics on the favorable chain are given by \(A_m\) in the ordered basis \((s_{m-1},s_{m-2},\dots,s_0)\). Equivalently, the restriction of \(M(m)\) to \(F_m\) is \(A_m\).
Indeed, let the parent token be \(s_i=(i,m)\). Since the layer is \(m\)-favorable, the new right defect is controlled by \(r_m\). If \(i\ge1\), the left child carries \((i-1,m)=s_{i-1}\), while the right child carries \((m-1,m)=s_{m-1}\). If \(i=0\), only the right child is guaranteed to remain in the chain, and it again carries \(s_{m-1}\). Writing these transitions column by column gives \(A_m\).

\paragraph{Example: the case \(m=2\).}
When \(m=2\), the favorable chain has states \(s_1=(1,2)\) and \(s_0=(0,2)\). For the most favorable symbol \(a=2\), one has \(M(2)e_{s_1}=e_{s_0}+e_{s_1}\) and \(M(2)e_{s_0}=e_{s_1}\). Hence, in the ordered basis \((s_1,s_0)\), the restricted matrix is
\[
A_2=
\begin{pmatrix}
1&1\\
1&0
\end{pmatrix}.
\]
Thus, when \(m=2\), the favorable chain is the Fibonacci transfer rule.

The bridge lemma below is the key structural input for lower bounds. It does
not make the true branch count multiplicative. Rather, it makes the certified
lower count quasi-multiplicative after the explicit insertion of one favorable
bridge symbol.

\begin{lemma}\label{thm:bridge-lemma}
For all words \(u,v\) in the coarse alphabet, we have
\(
\Gamma\bigl(v b^{(m)}u\bigr)\ge \Gamma(v)\Gamma(u),
	\)
with	\(
 b^{(m)}=(m).
\)
\end{lemma}

\begin{proof}
Write \(U=M(u)\) and \(V=M(v)\), and set \(x=P_{\mathrm{F}U_\star}\).
First note the elementary restart property of the bridge symbol: for every
favorable-chain state \(s_i=(i,m)\), the vector \(M(m)e_{s_i}\) contains at
least one copy of \(e_\star=e_{s_{m-1}}\). Indeed, if \(i\ge1\), the two
certified descendants include \((m-1,m)=s_{m-1}\), while for \(i=0\) the only
certified favorable descendant is again \(s_{m-1}\).
By this restart property and nonnegativity,
\(M(m)x\ge(\mathbf 1_{\mathrm F}^{\top}x)e_\star=\Gamma(u)e_\star\). The part
of \(Ue_\star\) outside \(\mathrm F_m\) can only add nonnegative mass, and
therefore \(M(m)Ue_\star\ge \Gamma(u)e_\star\). Applying \(P_{\mathrm F}\) and
then \(\mathbf 1_{\mathrm F}^{\top}\) after \(V\) gives
\[
\Gamma\bigl(v b^{(m)}u\bigr)
 =\mathbf 1_{\mathrm F}^{\top}P_{\mathrm F} V M(m)Ue_\star
 \ge \Gamma(u)\mathbf 1_{\mathrm F}^{\top}P_{\mathrm F}Ve_\star
 =\Gamma(u)\Gamma(v).
\]
\end{proof}

\section{Proof of Theorem \ref{thm:main-bridge-lower} -- A bridge lower bound}\label{sec:bridged-lower}

This section proves Theorem~\ref{thm:main-bridge-lower}. The proof uses the bridge lemma as a sufficient mechanism for forcing large values of the true branch count \(N_n\). The objects introduced below refer only to the certified lower process. They should not be read as exact descriptions of all branches.

Throughout this section we fix integers \(m,r\ge1\) and assume \(p_m>0\), equivalently \(p_b=p_m>0\). The parameter \(m\) fixes the defect resolution and the favorable bridge symbol, while \(r\) fixes the length of the free blocks.
We recall that
\(
\mathcal W_r:=\mathcal A_m^r,
\)
and
\(
\pi_r:=\pi^{\otimes r}.
\)
A block \(w\in\mathcal W_r\) is called productive if it leaves at least one certified favorable token:
\[
\mathcal G_r:=\{w\in\mathcal W_r:\Gamma(w)>0\}.
\]
For productive blocks we write \(g_r(w):=\log\Gamma(w)\). Thus, \(g_r\) is the logarithmic certified output of one free block.

The construction begins with one maximally favorable seed layer. This aligns the actual initial state \(e_{(m,m)}\) with the favorable-chain starting state \(e_\star\) used in the definition of \(\Gamma\).
For \(k\ge1\), define
\[
n_k:=1+kr+(k-1).
\]
The first symbol is a seed symbol. After that, the remaining \(n_k-1\)
coarse symbols are read as \(k\) free \(r\)-blocks
\(W_1,\dots,W_k\), separated by \(k-1\) copies of the fixed bridge
symbol \(b^{(m)}\), in the product order
\(
  W_k b^{(m)} W_{k-1} b^{(m)} \cdots b^{(m)} W_1 .
\)
Let
\[
L_k^{(r)}:=k^{-1}\sum_{i=1}^k\delta_{W_i}
\]
be the empirical law of the free blocks. Thus, \(L_k^{(r)}\) records how often each productive or nonproductive \(r\)-block appears in the construction. Let \(E_k\) be the event that the seed symbol is \(m\) and all \(k-1\) bridge slots equal the bridge word \(b^{(m)}\).

The probability cost of the forced seed and bridge slots is explicit.
Note that
for every \(k\ge1\), the event \(E_k\) is independent of \(L_k^{(r)}\), and
\(
\mathbb P(E_k)=p_mp_b^{k-1}.
\)
Indeed, 
the seed, free blocks, and bridge slots depend on disjoint groups of i.i.d.\ coarse symbols. Hence, the forced seed and bridge slots are independent of the free-block empirical law. The seed has probability \(p_m\), and each bridge slot has probability \(p_b\). There are \(k-1\) one-symbol bridge slots, so \(\mathbb P(E_k)=p_mp_b^{k-1}\).

We now prove the upper-tail lower bound for \(N_n\). Recall that \(J_{r,m}^{+}(x)\) is the optimal cost of choosing a free-block law \(\nu\) whose average certified output is large enough to exceed the target level \(x\), after accounting for the bridge cost \(p_b\). 
Here, Theorem~\ref{thm:main-bridge-lower} is a lower bound for the true observable \(N_n\), but it is obtained through the certified lower process. It may therefore be non-sharp if many true branches are not counted by \(\Gamma\).
The parameter \(r\) controls only the blocking of the free part of the word. For fixed \(m\), increasing \(r\) reduces the bridge overhead per unit depth, but it does not remove the possible gap between the certified count and the true branch count.

\begin{proof}[Proof of Theorem~\ref{thm:main-bridge-lower}]
If \(J_{r,m}^{+}(x)=\infty\), there is nothing to prove. Otherwise, fix
\(\varepsilon_0>0\). Choose \(\nu\in\mathcal P(\mathcal W_r)\), supported on
\(\mathcal G_r\), such that \(\int g_r\,d\nu>(r+1)x\) and
\((H(\nu\mid\pi_r)-\log p_b)/(r+1)\le J_{r,m}^{+}(x)+\varepsilon_0\).
Choose \(\alpha>0\) such that
\(\int g_r\,d\nu>(r+1)x+2\alpha\).

Since \(\mathcal W_r\) is finite and \(\nu\) is supported on \(\mathcal G_r\),
we may choose types \(\tau_k\in\mathcal P(\mathcal W_r)\), with denominator
\(k\), such that \(\tau_k\to\nu\), \(\tau_k(\mathcal G_r)=1\), and
\(\int g_r\,d\tau_k>(r+1)x+\alpha\) for all sufficiently large \(k\).

On the event \(E_k\cap\{L_k^{(r)}=\tau_k\}\), the seed event creates two
certified tokens in the top favorable state \(e_\star\). Moreover, all free
blocks belong to \(\mathcal G_r\). Hence repeated application of
Theorem~\ref{thm:bridge-lemma} to the forced word
\(W_k b^{(m)} W_{k-1} b^{(m)}\cdots b^{(m)}W_1\) gives
\[
N_{n_k}\ge 2\prod_{i=1}^k\Gamma(W_i),
\qquad
n_k^{-1}\log N_{n_k}
\ge n_k^{-1}\log 2+k n_k^{-1}\int g_r\,d\tau_k .
\]
Since \(n_k/k\to r+1\), the right-hand side is larger than \(x\) for all
sufficiently large \(k\). Thus
\(\mathbb P(n_k^{-1}\log N_{n_k}>x)\ge
\mathbb P(E_k\cap\{L_k^{(r)}=\tau_k\})\) for all sufficiently large \(k\).

By independence of the seed event, the bridge events, and the free blocks,
and by the standard method-of-types estimate on the finite alphabet
\(\mathcal W_r\),
\[
\mathbb P\bigl(E_k\cap\{L_k^{(r)}=\tau_k\}\bigr)
=
p_m p_b^{k-1}\,
\mathbb P\bigl(L_k^{(r)}=\tau_k\bigr),
\qquad
\lim_{k\to\infty}
k^{-1}\log \mathbb P\bigl(L_k^{(r)}=\tau_k\bigr)
=
-H(\nu\mid\pi_r).
\]
Using also \(n_k/k\to r+1\), the fixed seed cost \(\log p_m\) and the missing
single bridge factor are negligible after division by \(n_k\). Hence
\[
\liminf_{k\to\infty}
n_k^{-1}\log\mathbb P\bigl(n_k^{-1}\log N_{n_k}>x\bigr)
\ge
\frac{\log p_b-H(\nu\mid\pi_r)}{r+1}
\ge
-J_{r,m}^{+}(x)-\varepsilon_0 .
\]
Letting \(\varepsilon_0\downarrow0\) gives the claim.
\end{proof}

\section{Proofs of  Theorems \ref{thm:main-perron}, \ref{thm:main-all-favorable}, and \ref{thm:main-local} -- 
The uniformly favorable regime and local Telgarsky stability}\label{sec:uniform}

This section proves Theorems~\ref{thm:main-perron}, \ref{thm:main-all-favorable}, and \ref{thm:main-local}. The role of the uniformly favorable regime is to remove all probabilistic losses from the certified lower process: every layer is favorable, so the favorable chain evolves deterministically. Example~\ref{ex:fibonacci-level-two} is the first instance of this mechanism; the general level-\(m\) case has Perron rate \(\rho_m\uparrow2\).

First, the next deterministic estimate is the lower growth bound on the all-favorable event.

\begin{proposition}\label{prop:all-favorable-growth}
If the first \(n\) layers are all maximally favorable, then
\(
N_n\ge c_m\rho_m^n
\)
for some \(c_m>0\).
\end{proposition}

\begin{proof}
On the event that all first \(n\) coarse symbols are \(m\), the favorable chain evolves deterministically by \(A_m\). Starting from the initial interval, the first layer creates two favorable tokens in the top state, and afterwards the chain count is bounded below by \(\mathbf 1^\top A_m^{n-1}(2e_\star)\). By Perron--Frobenius theory for the primitive matrix \(A_m\), this quantity is at least \(c_m\rho_m^n\) after decreasing \(c_m\), if necessary, to handle finitely many small \(n\). Since these tokens correspond to actual branches, the same lower bound holds for \(N_n\).
\end{proof}

We now prove the all-favorable mechanism stated at the beginning of the paper.

\begin{proof}[Proof of Theorem~\ref{thm:main-all-favorable}]
Let \(E_n^{(m)}\) be the event that the first \(n\) layers are all maximally favorable. Then \(\mathbb P(E_n^{(m)})=p_m^n\). On this event, Proposition~\ref{prop:all-favorable-growth} gives \(N_n\ge c_m\rho_m^n\), and hence \(n^{-1}\log N_n\ge\log\rho_m+n^{-1}\log c_m\). For all sufficiently large \(n\), the right-hand side is at least \(\log\rho_m-\delta\), so
\[
\liminf_{n\to\infty} n^{-1}
\log\mathbb P\bigl(n^{-1}\log N_n\ge\log\rho_m-\delta\bigr)
\ge \log p_m.
\]
Moreover,
\(
\mathbb E[N_n]
\ge \mathbb E[N_n\mathbf 1_{E_n^{(m)}}]
\ge c_m\rho_m^n p_m^n
=c_m(p_m\rho_m)^n.
\)
The final claim is immediate.
\end{proof}

The next proposition quantifies how close the deterministic certified growth rate can be to the ideal Telgarsky rate \(2\). The parameter \(m\) is fixed throughout the proposition.

\begin{proposition}[Perron-root asymptotics]\label{prop:perron-root-asymptotics}
The Perron root \(\rho_m\) of \(A_m\) satisfies \(\rho_1=1\), \(\rho_m\in(1,2)\) for every \(m\ge2\), \(2-\rho_m=\rho_m^{-m}\), \(\rho_m\uparrow2\), \(2-\rho_m\sim2^{-m}\), and \(\log 2-\log\rho_m\sim2^{-m-1}\).
\end{proposition}

\begin{proof}
The matrix \(A_m\) is primitive: from any favorable-chain state one can reach the top state in one step, and from the top state one can reach any target state in at most \(m-1\) further steps. Hence, \(A_m^m\) is strictly positive. Its characteristic polynomial is \(\lambda^m-\lambda^{m-1}-\cdots-\lambda-1\), so its Perron root is the unique positive solution of
\(
\rho_m^m=\rho_m^{m-1}+\cdots+\rho_m+1	\).
Equivalently
	\(1=\sum_{j=1}^m\rho_m^{-j}.
\)
The case \(m=1\) gives \(\rho_1=1\). For \(m\ge2\), the function \(\lambda\mapsto\sum_{j=1}^m\lambda^{-j}\) is strictly decreasing on \((0,\infty)\), takes the value \(m>1\) at \(\lambda=1\), and takes a value strictly smaller than \(1\) at \(\lambda=2\). Thus, \(\rho_m\in(1,2)\). The same strict-monotonicity argument shows that \(\rho_{m+1}>\rho_m\), because \(\sum_{j=1}^{m+1}\rho_m^{-j}=1+\rho_m^{-(m+1)}>1\).

For \(m=1\), the identity \(2-\rho_m=\rho_m^{-m}\) is immediate. For \(m\ge2\), multiplying the characteristic equation by \(\rho_m-1\) gives \(\rho_m^{m+1}-2\rho_m^m+1=0\), which is equivalent to \(2-\rho_m=\rho_m^{-m}\). Since \(\rho_m\) is increasing and bounded above by \(2\), it has a limit. Moreover, \(\rho_m\ge\rho_2>1\) for \(m\ge2\), so \(\rho_m^{-m}\to0\). The identity \(2-\rho_m=\rho_m^{-m}\) therefore forces \(\rho_m\uparrow2\).
Write \(\delta_m:=2-\rho_m\). Since \(\delta_m=\rho_m^{-m}\le\rho_2^{-m}\), one has \(m\delta_m\to0\). Furthermore,
\(
\delta_m=\rho_m^{-m}
=2^{-m}(1-\delta_m/2)^{-m}.
\)
Since \(m\delta_m\to0\), it follows that \((1-\delta_m/2)^{-m}\to1\), and hence \(\delta_m\sim2^{-m}\). Finally, \(\log 2-\log\rho_m=-\log(1-\delta_m/2)\sim\delta_m/2\sim2^{-m-1}\).
\end{proof}

We now prove the favorable-chain theorem stated in Section~\ref{subsec:ldp-lower} and  the local lower-tail exclusion theorem. The latter is obtained by choosing the defect resolution first and the noise level second. This order is important. For a given tolerance \(\xi>0\), we first choose \(m\) so that the certified Perron rate \(\rho_m\) is within \(\xi\) of \(2\). Only after \(m\) is fixed do we choose \(\varepsilon\) small enough that every admissible layer is \(m\)-favorable.

\begin{proof}[Proof of Theorem~\ref{thm:main-perron}]
The assertions about the Perron root \(\rho_m\) are exactly Proposition~\ref{prop:perron-root-asymptotics}. The deterministic lower bound for \(N_n\) on the all-favorable event is Proposition~\ref{prop:all-favorable-growth}. Combining these two propositions proves the theorem.
\end{proof}

\begin{proof}[Proof of Theorem~\ref{thm:main-local}]
First choose \(m=m(\xi)\) so large that \(\log\rho_m>\log 2-\xi/2\), which is possible by Proposition~\ref{prop:perron-root-asymptotics}. 
Say that the environment is \(m\)-favorable if
\(
	\eta_*(\varepsilon):=\sup_{\omega\in U_{\varepsilon}} \eta(\omega)\le r_m .
\)
This condition says that the random environment is so close to the Telgarsky map
that every layer is favorable at level \(m\). 	With this \(m\) fixed, choose \(\varepsilon_\ast(\xi)>0\) so small that \(\eta_\ast(\varepsilon)\le r_m\) whenever \(0<\varepsilon\le\varepsilon_\ast(\xi)\). This is possible because \(\eta_\ast(\varepsilon)\to0\) as \(\varepsilon\downarrow0\), while \(r_m>0\) is now fixed.
Under this condition every layer is \(m\)-favorable, so Proposition~\ref{prop:all-favorable-growth} gives \(N_n\ge c_m\rho_m^n\) almost surely. Finally, choose \(n_0=n_0(\xi)\) so large that \(n^{-1}\log c_m\ge-\xi/2\) for all \(n\ge n_0\). Then, for all such \(n\) we have
\(
n^{-1}\log N_n
\ge \log\rho_m+n^{-1}\log c_m
> \log 2-\xi
\)
almost surely. Hence,
\(
\mathbb P\bigl(n^{-1}\log N_n\le \log 2-\xi\bigr)=0,
\)
for all \(n\ge n_0\).
\end{proof}

\section*{Declarations}
\noindent {\bf Acknowledgements.} The authors used OpenAI's ChatGPT during the preparation of this manuscript. All mathematical statements, proofs, computations, references, and conclusions were independently checked and verified by the authors, who take full responsibility for the content of the manuscript.

\medskip

\noindent {\bf Funding.} This research was supported by the Scientific and Technological Research Council of Türkiye (TÜBİTAK) under the TÜBİTAK-BİDEB 2219 International Postdoctoral Research Fellowship Programme, Grant No. 1059B192501091.
\medskip

\noindent {\bf Conflict of Interest}
On behalf of all authors, the corresponding author states that there is no conflict of interest. 
\medskip

\noindent {\bf Data Availability}
The data will be available on reasonable request.

% \medskip
% \noindent {\bf Conflict of interest.} Insert the appropriate Springer declaration before submission.

% \medskip
% \noindent {\bf Data availability.} Insert the appropriate Springer declaration before submission.

\bibliography{ll}

\end{document}